\documentclass[12pt]{amsart}

\usepackage[utf8]{inputenc}

\usepackage{amsfonts}
\usepackage{amsmath}
\usepackage{amssymb}
\usepackage{mathtools}
\usepackage{tabularx}

\usepackage{geometry} 
\usepackage{pxfonts}
\usepackage{euscript}
\usepackage{bbold,bbm}
\usepackage{comment}
\usepackage{scalerel}
\usepackage{xcolor}
\usepackage{tikz-cd}

\usepackage{hyperref}
\usepackage{graphics}
\usepackage{epstopdf} 
\usepackage[all,2cell,arrow,matrix,tips]{xy} \UseAllTwocells \SilentMatrices
\usepackage{graphicx}
\usepackage{verbatim}
\usepackage{leftidx}
\usepackage{enumerate}
\usepackage{phonetic}
\usepackage{lipsum}
\usepackage{subfigure}
\usepackage{float}

\usepackage[backend=biber]{biblatex}
\theoremstyle{definition}
\newtheorem{defn}{Definition}[section]
\newtheorem{thm}{Theorem}[section]

\newtheorem{lem}[thm]{Lemma}

\newcommand\Cb {\mathbb{C}}

\newcommand\CH {\EuScript{H}}

\newcommand\CK {\EuScript{K}}

\newcommand{\Cla} {\mathcal{A}}
\newcommand{\Clb} {\mathcal{B}}

\newcommand{\Clj} {\mathcal{J}}

\newcommand{\Cll} {\mathcal{L}}

\newcommand{\Alg}{\operatorname{Alg}}
\newcommand{\Lat}{\operatorname{Lat}}

\DeclareMathOperator{\Span}{Span}

\DeclarePairedDelimiterX\braket[2]{\langle}{\rangle}{#1 \delimsize\vert #2}

\begin{document}

\title{Non-Existence of Abelian KS-Algebras in $M_n(\Cb)$ for $n \geq 3$}

\author{Wei Yuan}
\address{Institute of Mathematics, Academy of Mathematics and Systems Science, Chinese Academy of Sciences, Beijing, 100190, China; School of Mathematical Sciences, University of Chinese Academy of Sciences, Beijing 100049, China.}
\email{wyuan@math.ac.cn}

\date{}

\begin{abstract}
We prove that $M_n(\Cb)$ contains no abelian Kadison--Singer algebra for
any $n \geq 3$. 
\end{abstract}


\maketitle

\section{Introduction}
Let $\Clb(\CH)$ denote the algebra of bounded linear operators on a complex Hilbert space $\CH$. For a subalgebra $\Cla \subseteq \Clb(\CH)$, let $\Lat\Cla$ denote the lattice of its closed invariant subspaces, and let $\Alg\Lat\Cla$ denote the algebra of all bounded linear operators that leave every subspace in $\Lat\Cla$ invariant. The algebra $\Cla$ is called reflexive if $\Cla = \Alg\Lat\Cla$.

In 2010, the notion of Kadison--Singer algebras was introduced by Ge and Yuan to synthesize ideas from triangular operator algebras, reflexive algebras, and von Neumann algebras \cite{GY1}.

\begin{defn}[\cite{GY1}]
A subalgebra $\Cla \subseteq \Clb(\CH)$ is a \emph{Kadison--Singer algebra} (or \emph{KS-algebra}) if it is reflexive and maximal with respect to inclusion among reflexive algebras sharing the same diagonal $\Cla \cap \Cla^*$, where $\Cla^* = \{A^* : A \in \Cla\}$. Explicitly, if $\Clb \subseteq \Clb(\CH)$ is any reflexive subalgebra satisfying
\begin{align*}
  \Cla \subseteq \Clb
  \quad\text{and}\quad
  \Clb \cap \Clb^* = \Cla \cap \Cla^*,
\end{align*}
then $\Clb = \Cla$.
\end{defn}

In \cite{Yuan14}, Wu and Yuan initiated the study of abelian KS-algebras and posed the question of whether examples exist beyond the known one-parameter family in $M_2(\Cb)$:
\begin{align*}
  \left\{
    \begin{pmatrix}
      a & t(b-a) \\
      0 & b
    \end{pmatrix}
    : a, b \in \Cb
  \right\}, \quad t > 0.
\end{align*}
They proved that, for $n \ge 3$, an abelian KS-algebra in $M_n(\Cb)$ contains no non-trivial idempotents, and consequently must be of the form $\Cb I + \Clj$, where $I$ is the identity matrix and $\Clj$ is a nilpotent ideal \cite[Corollaries~2.3--2.4]{Yuan14}. Using this characterization, they ruled out the existence of such algebras generated by a single matrix. 

The purpose of this short note is to remove this restriction on the number of generators and answer the question in the negative for all dimensions $n \ge 3$. More explicitly, we establish the following theorem:

\begin{thm} \label{thm:no_ab_KS_3}
For $n \ge 3$, there exists no abelian Kadison--Singer algebra in $M_n(\Cb)$.
\end{thm}

\subsection*{Declaration on the Use of AI and Waiver of Credit}
An early version of the proof was developed with the assistance of the large language model GPT-6. The author framed the intermediate lemmas, directed the reasoning steps through interactive prompting, and carefully checked and revised the final arguments. The author assumes full responsibility for the mathematical correctness of the text.

This preprint is circulated solely to share these findings with the mathematical community, with no intention of claiming priority. First, the primary conceptual groundwork was established in \cite{Yuan14}, making the author’s contribution essentially dispensable; second, the interactive discovery process felt largely akin to a lottery.

\section{Proof of Theorem \ref{thm:no_ab_KS_3}}

Throughout this section, closed subspaces of a Hilbert space $\CH$ are canonically identified with the orthogonal projections onto them. For an unital subalgebra $\Cla$ of $\Clb(\CH)$ and $\xi \in \CH$, let
\begin{align*}
 \Cla \xi := \Span \{A \xi : A \in \Cla\},
\end{align*}
and let $\overline{\Cla\xi}$ denote its closure in $\CH$. In particular, $\overline{\Cla \xi} = \Cla \xi$ whenever $\dim \CH < \infty$.\\

The following description of $\Alg \Lat \Cla$ is well-known. We include the proof for the sake of completeness.

\begin{lem}\label{lem:local_ref}
Let $\Cla$ be a unital subalgebra of $\Clb(\CH)$ and $T \in \Clb(\CH)$. Then 
\begin{align*}
  T \in \Alg \Lat \Cla
\end{align*}
if and only if 
\begin{align*}
  T \xi \in \overline{\Cla \xi} \quad \text{for all } \xi \in \CH.
\end{align*}
\end{lem}

\begin{proof}
For any $\xi \in \CH$, the closed subspace $\overline{\Cla \xi}$ is invariant under $\Cla$, so it belongs to $\Lat \Cla$. Consequently, if $T \in \Alg \Lat \Cla$, then $\overline{\Cla \xi}$ is invariant under $T$. Because $\Cla$ is unital, $\xi \in \overline{\Cla \xi}$, which implies $T\xi \in \overline{\Cla \xi}$.

Conversely, assume that $T\xi \in \overline{\Cla \xi}$ for every $\xi \in \CH$. Let $\CH_0 \in \Lat \Cla$ and choose any $\beta \in \CH_0$. Since $\CH_0$ is invariant under $\Cla$, we have $\overline{\Cla \beta} \subseteq \CH_0$. Thus,
\begin{align*}
  T\beta \in \overline{\Cla \beta} \subseteq \CH_0,
\end{align*}
which shows that $\CH_0$ is invariant under $T$. Thus $T \in \Alg \Lat \Cla$.
\end{proof}

\begin{lem}\label{lem:con_cut}
Suppose that $\Cla = \Cb I + \Clj \subseteq M_n(\Cb)$, $n \geq 2$, is a commutative reflexive algebra, where $\Clj$ is an ideal of $\Cla$ satisfying $\Clj^m = \{0\}$ for some integer $m \ge 1$. Define the subspace
\begin{align*}
  \CK := \Span \{A \xi : A \in \Clj, \, \xi \in \Cb^n\} \subseteq \Cb^n.
\end{align*}
Let $P \in M_n(\Cb)$ be the orthogonal projection onto $\CK$. Treating $P \Cla P$ as a unital subalgebra of $P M_n(\Cb) P$, we have
\begin{align*}
  \Alg\Lat(P \Cla P) \subseteq (P \Cla P)' \cap P M_n(\Cb) P
\end{align*}
and
\begin{align*}
  (\Lat P \Cla P)' \cap P M_n(\Cb) P = \Cb P.
\end{align*}
\end{lem}

\begin{proof}
It is clear that $\CK$ is an invariant subspace of $\Cla$. Let $T \in \Alg\Lat(P \Cla P)$. For every $J  \in \Clj$ and $\xi \in \Cb^n$, $J \xi \in \CK$. By Lemma \ref{lem:local_ref},
\begin{align*}
 TJ \xi \in P\Cla P J \xi = \Cla J \xi \subseteq \Cla \xi, \qquad \forall \xi \in \Cb^n.
\end{align*}
Then Lemma \ref{lem:local_ref} implies that $TJ \in \Cla$. For every $A \in \Cla$, we have
\begin{align*}
 \bigl(TPAP- PAPT\bigr)J \xi = (TJA- ATJ)\xi=0, \qquad \forall \xi \in \Cb^n.
\end{align*}
Therefore $T \in (P \Cla P)' \cap P M_n(\Cb) P$. 

We now verify that $(\Lat P\Cla P)' \cap P M_n(\Cb) P = \Cb P$. Let $Q \in (\Lat P\Cla P)' \cap P M_n(\Cb) P$ be a projection. It suffices to show that $Q = 0$ or $Q = P$. The discusssion above ensures that $Q \in (P \Cla P)' \cap P M_n(\Cb) P$.

Assume that both $Q$ and $P - Q$ are nonzero. Then
\begin{align*}
  Q, P-Q \in \Lat \Cla.
\end{align*} 
We claim that for every nonzero $E \in \Lat(P\Cla P)$, there exists a nonzero vector $\beta \in E \Cb^n$ such that $\Clj \beta = \{0\}$. Indeed, let $k \ge 0$ be the integer such that 
\begin{align*}
  \Clj^k E \neq \{0\} \quad \text{and} \quad \Clj^{k+1} E = \{0\},
\end{align*}
with the convention that $\Clj^0 E = E$. Any nonzero vector $\beta \in \Clj^k E \Cb^n$ then satisfies the desired condtions. 

By the above claim, there exist nonzero vectors $\beta_1 \in Q \Cb^n$ and $\beta_2 \in (P - Q) \Cb^n$ such that 
\begin{align*}
 \Clj \beta_1 = \{0\}, \qquad \Clj \beta_2 = \{0\}.  
\end{align*}
The one-dimensional subspace $\Cb(\beta_1 + \beta_2)$ is invariant under $\Cla$ because
\begin{align*}
  \Clj (\beta_1 + \beta_2) = \Clj \beta_1 + \Clj \beta_2 = \{0\}.
\end{align*}
However, this contradicts the assumption that $Q \in \Alg \Lat(P \Cla P)$, since
\begin{align*}
  Q(\beta_1 + \beta_2) = \beta_1 \notin \Cb(\beta_1 + \beta_2).
\end{align*}
We thus conclude that $Q = 0$ or $Q = P$.
\end{proof}

\begin{lem}\label{lem:nil_not_KS}
Under the notation and assumptions of Lemma~\ref{lem:con_cut}, $\Cla$ is not a KS-algebra.
\end{lem}

\begin{proof}
If $\Clj = \{0\}$, let
\begin{align*}
   F = \frac{1}{n} \sum_{i,j} E_{ij},
\end{align*}
and define
\begin{align*}
  \Cll := \{E_{11}, E_{11} + E_{22}, \ldots, E_{11} + \cdots + E_{nn}, F\},
\end{align*}
where $\{E_{in}\}_{ij}$ are the canonical mattrix units of $M_n(\Cb)$. Then $\Alg \Cll$ is a reflexive algebra strictly containing $\mathbb{C} I$ with
\begin{align*}
(\Alg \Cll)^* \cap \Alg \Cll = \mathbb{C} I.
\end{align*}
In particular, $\Cla = \Cb I$ is not a KS-algebra.

Assume now that $\Clj \neq \{0\}$ and $\Cla$ is a KS-algebra. Note that $\CK \neq \Cb^n$. Then $I-P \neq 0$. Since $\Cla = \Cb I + \Clj$, $P \notin \Cla$. By Lemma \ref{lem:local_ref}, there exists a nonzero vector $\xi \in \Cb^n$ such that
\begin{align*}
  P \xi \notin \Cla \xi.
\end{align*}
In particular, $P \xi \neq 0$ and $(I-P) \xi \neq 0$ since $0, \xi \in \Cla \xi$. Let $Q$ be the projection onto $\Cla \xi$ and 
\begin{align*}
 \Cll :=\{E \leq P: E \in \Lat(P \Cla P) \} \cup \{P +  F:  F \leq I-P\} \cup \{Q\}.
\end{align*}
Since $P\CH = \Clj \CH$, $\Cll \subseteq \Lat \Cla$.

We claim that $\Cll' = \Cb I$. Indeed, let $T \in \Cll'$. By Lemma~\ref{lem:con_cut}, there exist $a, b \in \Cb$ such that 
\begin{align*}
  T = aI + bP.
\end{align*}
Since $PQ \neq QP$, we conclude that $T = aI$. Therefore
\begin{align*}
  \Cla^* \cap \Cla = \Cb I = (\Alg \Cll)^* \cap \Alg \Cll.
\end{align*}

It remains to show that 
\begin{align*}
  \Cla \subsetneq \Alg(\Cll),
\end{align*}
which in particular implies that $\Cla$ is not a KS-algebra. Let $B \in M_n(\Cb)$ satisfy
\begin{align*}
  B = PB(I-P), \quad \text{and} \quad B(I-P) \xi = P\xi.
\end{align*}
Then $B + (I-P)$ is an idempotent. Since 
\begin{align*}
  [B + (I-P)] \Clj = [B + (I-P)] P\Clj = \{0\},
\end{align*}
we have 
\begin{align*}
  [B + (I-P)] \Cla \xi = [B + (I-P)] (\Cb \xi + \Clj \xi) = [B + (I-P)] \Cb \xi \subseteq \Cla \xi,
\end{align*}
for the fixed vector $\xi$ chosen above. It is also straightforward to verify that 
\begin{align*}
  \{E \le P : E \in \Lat(P\Cla P)\} \cup \{P + F : F \le I-P\} \subseteq \Lat [B + (I-P)].
\end{align*}
Therefore, $B + (I-P) \in \Alg \Cll$. It is clear that the idempotent $B + (I-P)$ does not belong to $\Cla = \Cb I + \Clj$. This completes the proof of the lemma.
\end{proof}

\noindent\textbf{Proof of Theorem \ref{thm:no_ab_KS_3}}: Assume that $\Cla \subseteq M_n(\Cb)$ is an abelian KS-algebra. By \cite[Corollary~2.4]{Yuan14}, $\Cla = \Cb I + \Clj$, where $\Clj$ is an abelian algebra consisting entirely of nilpotent matrices. Since commuting complex matrices admit a common eigenvector, we may assume, up to unitary equivalence, that every element in $\Clj$ is strictly upper triangular. In particular, $\Clj^n = \{0\}$. By Lemma~\ref{lem:nil_not_KS}, this yields a contradiction, which proves that no abelian KS-algebra exists in $M_n(\Cb)$, $n \geq 3$.

\printbibliography
\end{document}